\documentclass{ws-ccm}

\usepackage{cleverref}
\usepackage{dsfont}
\usepackage{amssymb,esint}

\newcommand{\C}{\kappa}

\newcommand{\hb}{\hat b}

\newcommand{\mN}{\mathbb{N}}

\newcommand{\mZ}{\mathbb{Z}}

\newcommand{\mR}{\mathbb{R}}

\newcommand{\eps}{\varepsilon}

\newcommand{\dsp}{\displaystyle}

\newcommand{\bb}{\gamma}

\numberwithin{equation}{section}

\begin{document}

\title{\Large $\Gamma$-convergence of non-local, non-convex functionals in one dimension}


\author{HA\"IM  BREZIS} 
\address{Department of Mathematics, 
Rutgers University, Hill Center, Busch Campus, 
	110 Frelinghuysen Road, Piscataway, NJ 08854, USA,\\
	 Departments of Mathematics and Computer Science, Technion, Israel Institute of Technology, 32.000 Haifa, Israel, \\
Laboratoire Jacques-Louis Lions, 
Sorbonne Universit\'es, UPMC Universit\'e Paris-6, 4  place Jussieu, 
75005 Paris, France,\\
brezis@math.rutgers.edu}

\author{HOAI-MINH NGUYEN}

\address{Department of Mathematics, EPFL SB CAMA, 
	Station 8 CH-1015 Lausanne, Switzerland, 
	hoai-minh.nguyen@epfl.ch}

\maketitle

\begin{history}
\received{(Day Month Year)}
\revised{(Day Month Year)}
\end{history}

\begin{abstract} We study the  $\Gamma$-convergence of a family of non-local, non-convex  functionals in $L^p(I)$ for $p \ge 1$, where $I$ is an open interval.  We show that the limit is a multiple of  the $W^{1, p}(I)$ semi-norm to the power $p$ when $p>1$ (resp. the $BV(I)$ semi-norm when $p=1$). In dimension one, this extends earlier results which required a monotonicity condition. 

\end{abstract}







\section{Introduction and statement of the main results}

Assume that  $\varphi:[0, +\infty) \to [0, + \infty)$ is defined at {\it every} point of $[0, + \infty)$, $\varphi$ is
 continuous on $[0, +\infty)$ except at a finite number of points in $(0, +\infty)$ where it admits a limit from the left and from the right, and $\varphi(0) = 0$. Let $I$ denote an  open interval of $\mR$.  Fix $p \ge 1$. Given a measurable function $u$ on $I$, and a  parameter $\delta > 0$, we define, as in \cite{BrNg-Nonlocal1},  the following non-local functionals
\begin{equation}\label{def-Lambda}
\Lambda (u, I): = \int_I \int_I \frac{\varphi(|u(x) - u(y)|) }{|x - y|^{p + 1}} \, dx \,
dy  \le + \infty 
\end{equation}
and 
$$
 \Lambda_{\delta}(u, I): = \delta^p \Lambda (u/\delta, I).
$$
Throughout the paper,  we make the following three assumptions on $\varphi$:
\begin{equation}\label{cond-varphi-0}
\varphi(t) \le \alpha t^{p+1} \mbox{ on } [0,1] \mbox{ for some positive constant } \alpha,
\end{equation}
\begin{equation}\label{cond-varphi-1}
\varphi(t) \le \beta  \mbox{ on } [0, + \infty) \mbox{ for some positive constant } \beta,
\end{equation}
and
\begin{equation}\label{cond-varphi-3}
\int_0^\infty \varphi(t) t^{-(p+1)} \,dt=1/2. 
\end{equation}

%
%
%


\medskip 

 Our main result is the following  

\begin{theorem}\label{thm-gamma} Let $p \ge 1$ and let $\varphi$ satisfy \eqref{cond-varphi-0}-\eqref{cond-varphi-3}. Then, as $\delta \to 0$,
\begin{equation*}
\Lambda_{\delta}(\cdot,  I) \; \Gamma\mbox{-converges in $L^p(I)$ to }  \Lambda_0(\cdot, I), 
\end{equation*}
where 
\begin{equation}\label{def-Lambda0}
 \Lambda_0(u, I) = \C \int_{I} |u' |^p \, dx \mbox{ in } L^p(I), 
\end{equation}
for some constant $\C$, depending on $\varphi$ but independent of $I$, such that 
\begin{equation}\label{constant-gamma}
0 \le   \C \le 1.  
\end{equation}
\end{theorem}


%

\medskip
Some comments on Theorem~\ref{thm-gamma} are in order.


$\bullet$ On the precise definition of $\Lambda_0$. If $\C = 0$, by convention, $\Lambda_0(u, I) = 0$ for all $u \in L^p(I)$. 
In other words, the conclusion of \Cref{thm-gamma} asserts that {\it either} $\Lambda_\delta(\cdot, I)$ $\Gamma$-converges to 0 in $L^p(I)$ {\it or} there exists a constant $0< \C \le 1$ such that $\Lambda_\delta(\cdot, I)$ $\Gamma$-converges to $\Lambda_0(\cdot, I)$ defined by \eqref{def-Lambda0} with the  usual convention: $\Lambda_0(u, I) = + \infty$ if $u \not \in BV(I)$ for $p=1$, or if $u \not \in W^{1, p}(I)$ for $p>1$. The first part of the alternative, i.e., $\C = 0$, occurs  e.g. when 
$\varphi$ has a compact support in $[0, + \infty)$ (see \cite[Remark 3]{BrNg-Nonlocal1}; only the case $p=1$ was considered in \cite{BrNg-Nonlocal1},  however, the same conclusion holds for $p>1$ with the same proof).  The second part of the alternative, i.e., $\C>0$, happens e.g. when $\varphi$ is non-decreasing  (see  \cite{BrNg-Nonlocal1, BrNg-Nonlocalp} with roots in \cite{BourNg}); more generally, $\C > 0$ when $\liminf_{t \to + \infty} \varphi(t) > 0$. It would be very interesting to find a natural weaker sufficient condition on $\varphi$ at infinity such that $\C>0$. 

$\bullet$ 
On the condition~\eqref{cond-varphi-3}. This is just a normalization condition. Without this assumption, the conclusion of Theorem~\ref{thm-gamma} holds with \eqref{constant-gamma} replaced by 
$$
0 \le \C \le 2  \int_0^{\infty} \varphi(t) t^{-(p+1)} \, dt. 
$$
This suggests that assumptions \eqref{cond-varphi-0}-\eqref{cond-varphi-1} might be substituted by the weaker condition 
$$
\int_0^{\infty} \varphi(t) t^{-(p+1)} \, dt < + \infty. 
$$


It is worth noting that the following pointwise convergence property 
holds for $\Lambda_\delta$:

\begin{proposition}\label{pro-pointwise} Let $p \ge 1$ and let $\varphi$ satisfy \eqref{cond-varphi-0}-\eqref{cond-varphi-3}. Then,  
\begin{enumerate}
\item[i)] for $p>1$ and for $u  \in W^{1, p}(I)$, 
\end{enumerate}
or
\begin{enumerate}
\item[ii)] for $p=1$ and  for $u \in C^1(\bar I)$ if $I$ is bounded (resp. $u \in C^1_c(\bar I)$ if $I$ is unbounded), 
\end{enumerate}
we have
\begin{equation}\label{pro-pointwise-s}
\lim_{\delta \to 0} \Lambda_{\delta} (u, I) = \int_{I} |u' |^p \, dx. 
\end{equation} 
\end{proposition}

The conclusion  of \Cref{pro-pointwise} under the assumption $i)$ follows from \cite[Theorem 1]{BrNg-Nonlocalp} (the only remaining case to be considered is the case $I = (0, + \infty)$ which can be deduced from the cases $I$ bounded and $I = \mR$ by standard arguments). 
The proof of  \Cref{pro-pointwise} under the assumption $ii)$ appeared in \cite[proof of Proposition 1]{BrNg-Nonlocal1}
under the additional assumption
\begin{equation}\label{cond-varphi-mono}
\mbox{$\varphi$ is non-decreasing};    
\end{equation} 
however, this assumption can  be easily  removed from the proof.  The conclusion of \Cref{pro-pointwise} contrasts with the conclusion of Theorem~\ref{thm-gamma} since 
it may happen, for some functions $\varphi$ $\big(\mbox{e.g. } \varphi : = \frac{p}{2} \mathds{1}_{(1,  + \infty)}\big)$,  that  $\C$ is {\it strictly} less than 1  (see \cite{NgGammaCRAS}); an explicit value of $\C$ for this $\varphi$ is given in \cite{AGMP}.  As  established in \cite{AGP}, it may happen  that $\C(\varphi) = 1$ for some $\varphi$. 

This work is a follow-up of  our previous papers \cite{BrNg-Nonlocal1,BrNg-Nonlocalp} where we investigated   a similar problem in any dimension $d \ge 1$. More precisely, $I$ is replaced by a domain $\Omega \subset \mR^d$ 
and the RHS in \eqref{def-Lambda} is replaced by 
$$
\int_\Omega \int_\Omega \frac{\varphi(|u(x) - u(y)|) }{|x - y|^{p  + d}} \, dx \,
dy.
$$
Assuming  \eqref{cond-varphi-0}, \eqref{cond-varphi-1}, and the {\it additional condition} \eqref{cond-varphi-mono},  we established in \cite{BrNg-Nonlocal1,BrNg-Nonlocalp} the $\Gamma$-convergence of $\Lambda_\delta$ to a multiple of $\int_{\Omega} |\nabla u|^p \, dx$.  
In these works, the monotonicity assumption \eqref{cond-varphi-mono} played a crucial role at almost every level of the proofs. The proof of \Cref{thm-gamma} has its roots in \cite{BourNg, NgGamma, BrNg-Nonlocal1, BrNg-Nonlocalp}.  However, many  new ideas are required to overcome the lack of assumption  \eqref{cond-varphi-mono}. We do not know whether  \eqref{cond-varphi-mono} can be removed when $d > 1$.  


\section{Proof of the main result}\label{sect-proof}

We first recall  the meaning of $\Gamma$-convergence. One says  that  $\Lambda_\delta(\cdot, I)$ $\dsp \mathop{\to}^{\Gamma} \Lambda_0(\cdot, I)$ in $L^p(I)$ for $p \ge 1$ as $\delta \to 0$ if  the following two properties hold 
\begin{enumerate}
\item[(G1)] For each $g \in L^p(I)$ and for {\it every} family
$(g_\delta) \subset L^p(I)$ such that
$g_\delta$ converges to $g$ in $L^p(I)$ as $\delta \to 0$,
one has
\begin{equation*}
\liminf_{\delta \to 0} \Lambda_\delta(g_\delta, I) \ge \Lambda_0(g, I). 
\end{equation*}
\item[(G2)] For each $g \in L^p(I)$, there {\it exists} a family
$(g_\delta) \subset L^p(I)$ such that
$g_\delta$ converges to $g$ in $L^p(I)$ as $\delta \to 0$,
and
\begin{equation*}
\limsup_{\delta \to 0} \Lambda_\delta(g_\delta, I) \le \Lambda_0(g, I).
\end{equation*}
\end{enumerate}


In this section, we establish  properties (G1) and (G2)  with $\Lambda_0$ defined by \eqref{def-Lambda0} and  $\C$ defined by 
\begin{equation}\label{def-k}
\C :=  \inf \liminf_{\delta \to 0} \Lambda_{\delta} (v_\delta, (0, 1)),
\end{equation}
where the infimum is taken over all families $(v_\delta) \subset L^p(0, 1)$ such that $v_\delta \to U$ in $L^p(0, 1)$ as $\delta \to 0$, where 
$$
U(x) := x \mbox{ for } x \in (0, 1).
$$

Choosing $I = (0, 1)$ and $u = U$ in  \Cref{pro-pointwise}, we see  that  the constant $\C$ given by \eqref{def-k} satisfies 
$0 \le \C \le 1$. 

\begin{remark} \label{rem-k} \rm As a direct consequence of the definition of $\C$ in  \eqref{def-k}, the following property holds
$$
\liminf_{k \to + \infty} \Lambda_{\delta_k}(g_k, (0, 1)) \ge \C, 
$$
for every $(\delta_k) \subset \mR_+$ and $(g_k) \subset L^p(0, 1)$ such that $\delta_k \to 0$ and $g_k \to U$ in $L^p(0, 1)$ as $k \to + \infty$. 
\end{remark}

We will only consider the case $I = \mR$. The other cases can be handled as in \cite{BrNg-Nonlocal1} and are left to the reader. The rest of this paper is organized as follows. \Cref{sect-G2} is devoted to the proof of Property (G2). The proofs of Property (G1) for $p=1$ and $p>1$ are given in \Cref{sect-G1-p,sect-G1-1}, respectively. 

For $p \ge 1$ and  $\delta > 0$, we will denote
$$
\varphi_\delta(t) := \delta^p \varphi(t/\delta) \mbox{ for } t \ge 0. 
$$

\subsection{Proof of Property (G2)}\label{sect-G2}

The proof of Property (G2) is based on the following three lemmas which are valid for $\C$ defined by \eqref{def-k}, possibly equal to $0$. We begin with  

\begin{lemma}\label{lem1} 
Let $p \ge 1$ and let $\varphi$ satisfy \eqref{cond-varphi-0}-\eqref{cond-varphi-3}.
There exists a family  $(v_\delta) \subset L^p(0, 1)$ converging to $U$  in $L^p(0, 1)$,  as $\delta \to 0$,  such that 
\begin{equation}\label{lem1-statement}
\lim_{\delta \to 0} \Lambda_{\delta} (v_\delta, (0, 1)) = \C. 
\end{equation}
\end{lemma}

\begin{proof}  
From the definition of $\C$ in \eqref{def-k}, there exist a
sequence $(\delta_k) \subset \mR_+$ converging to 0  and a sequence  $(u_{k}) \subset L^p(0, 1)$
converging to $U$ in $L^p(0, 1)$ such that
\begin{equation}\label{lem1-start}
\lim_{k \to + \infty} \Lambda_{\delta_k} (u_k, (0, 1)) = \C.
\end{equation}
Let $(c_k)$ be a sequence of positive numbers converging to $0$ such that, for large $k$,  
\begin{equation}\label{lem1-p1}
c_k \ge \delta_k^{1/2}, 
\end{equation}
\begin{equation}\label{lem1-p2}
\int_{0}^1 |u_k - U|^p \, dx  \le c_k^{p+1}, 
\end{equation}
\begin{equation}\label{lem1-p3}
\Lambda_{\delta_{k}}(u_k, (0, 1)) \le \C + c_{k}, 
\end{equation}
\begin{equation}\label{lem1-p4}
\Lambda_{\delta_{k}}(u_k, (c_k, 1 - c_k)) \ge \C(1- 2 c_k) - c_k. 
\end{equation}
Such a sequence $(c_k)$ exists; indeed, from the definition of $\C$,  by a change of variables, we have
\begin{equation*}
\liminf_{k \to + \infty} \Lambda_{\delta_k}(u_k, (c, 1 - c)) \ge \C(1- 2 c), 
\end{equation*} 
for every $c \in (0, 1/2)$.  Hereafter, we only consider large $k$ so that \eqref{lem1-p1}-\eqref{lem1-p4} hold. 

In what follows in this proof, $C$ denotes  positive  constants which depend only on $\alpha$,  $\beta$, and $p$ and can vary from one place to another. From  \eqref{lem1-p3}
and \eqref{lem1-p4}  and the fact that $\C \le 1$, 
we obtain 
\begin{equation}\label{est-mean}
\int_{c_k/2}^{c_k} \int_0^1 \frac{\varphi_{\delta_k}(|u_k(x) - u_k(y)|)}{|x- y|^{p+1}} \, dx \, dy \le C c_k. 
\end{equation} 
By \eqref{lem1-p2} and \eqref{est-mean},  there exists 
\begin{equation}\label{lem1-x1k}
x_{1, k} \in (c_k/2, c_k)
\end{equation} 
such that
\begin{equation}\label{lem1-o1}
|u_k(x_{1, k}) - x_{1, k}| \le C c_k
\end{equation} 
and
\begin{equation}\label{lem1-o2}
\int_0^1 \frac{\varphi_{\delta_k}(|u_k(x_{1,k}) - u_k(y)|)}{|x_{1, k}- y|^{p+1}}  dy \le C.  
\end{equation} 
Similarly, there exists 
\begin{equation}\label{lem1-x2k}
x_{2, k} \in (1 - c_k, 1- c_{k}/2)
\end{equation} 
such that 
\begin{equation}\label{lem1-o3}
|u_k(x_{2, k}) - x_{2, k}| \le C c_k
\end{equation} 
and
\begin{equation}\label{lem1-o4}
\int_0^1 \frac{\varphi_{\delta_k}(|u_k(x_{2,k}) - u_k(y)|)}{|x_{2, k}- y|^{p+1}}  dy  \le C.  
\end{equation} 

We now modify $u_k$ to obtain a new sequence $(\hat u_k)$ such that $\hat u_k \to U$ in $L^p(0, 1)$, 
\eqref{lem1-start} is preserved for $\hat u_k$, i.e., 
\begin{equation}\label{lem1-start-1}
\lim_{k \to + \infty} \Lambda_{\delta_k} (\hat u_k, (0, 1)) = \C,  
\end{equation}
and {\it in addition} 
$$
\hat u_k = U \mbox{ in suitable neighborhoods of 0 and 1}. 
$$
Define $\hat u_k: (0, 1 ) \to \mR$ as follows
\begin{equation*}
\hat u_k (x)  : = \left\{\begin{array}{cl}
x& \mbox{ if } 0  <   x < \frac{x_{1, k}}{3}, \\[6pt]
u_k(x_{1, k}) & \mbox{ if } \frac{2x_{1, k}}{3} <   x < x_{1, k}, \\[6pt]
u_k(x) & \mbox{ if } x_{1, k} \le x \le x_{2, k}, \\[6pt]
u_k(x_{2, k}) & \mbox{ if } x_{2, k} < x <  \frac{1+2x_{2, k}}{3}, \\[6pt]
x & \mbox{ if } \frac{2+x_{2, k}}{3} < x <  1, 
\end{array}\right.
\end{equation*}
and $\hat u_k$ is chosen in $[\frac{x_{1,k}}{3}, \frac{2 x_{1, k}}{3}] \cup [ \frac{1+2x_{2, k}}{3}, \frac{2+x_{2, k}}{3}]$ in such a way that it is affine there and $\hat u_k$ is  continuous at the end points.

We claim that 
\begin{equation}\label{lem1-e}
\Lambda_{\delta_k} (\hat u_k, (0, 1)) \le \C + C c_k. 
\end{equation}
For this purpose, we  estimate  $\Lambda_{\delta_k} (\hat u_k, (0, 1))$ writing 
\begin{align}\label{lem1-decomp}
 \Lambda_{\delta_k} & (\hat u_k, (0, 1))  \\[6pt]
 \le &  \Lambda_{\delta_k}(\hat u_k, (0,  x_{1, k})) + 2
\int_{2x_{1,k}/3}^{x_{1, k}} \, dx  \int_{x_{1, k}}^{x_{2, k}} \frac{\varphi_{\delta_k}(|\hat u_k(x) - \hat u_k(y)|)}{|x - y|^{p+1}}  dy  \nonumber \\[6pt]
& + \Lambda_{\delta_k}(\hat u_k, (x_{2, k},  1)) 
+ 2  \int_{x_{2,k}}^{ (1 + 2x_{2, k})/3}  \int_{x_{1, k}}^{x_{2, k}} \frac{\varphi_{\delta_k}(|\hat u_k(x) - \hat u_k(y)|)}{|x - y|^{p+1}}  dy \, dx  \nonumber \\[6pt]
& +  \Lambda_{\delta_k}(\hat u_k, (x_{1, k},  x_{2, k}))   + \mathop{\int_0^1 \int_0^1}_{|x - y| > \min\{x_{1, k}/3, (1- x_{2, k})/3 \}} \frac{\varphi_{\delta_k}(|\hat u_k(x) - \hat u_k(y)|)}{|x - y|^{p+1}}  dy \, dx. \nonumber  \\[6pt]
& : = I + II + III + IV + V + VI.  \nonumber
\end{align}
We begin with $I$. We have, by \eqref{cond-varphi-0} and \eqref{cond-varphi-1} 
\begin{align*}
I = &  \Lambda_{\delta_k}(\hat u_k, (0,  x_{1, k}))  = \int_0^{x_{1, k}} \int_0^{x_{1, k}} \frac{\varphi_{\delta_k}(|\hat u_k(x) - \hat u_k(y)|) }{|x - y|^{p + 1}} \, dx \, dy \\[6pt]
\le &  C \mathop{\int_0^{x_{1, k}} \int_0^{x_{1, k}}}_{|x - y | \le \delta_k}   \frac{ \delta_k^{p} |\hat u_k(x) - \hat u_k(y)|^{p+1}}{\delta_k^{p+1}} \frac{1}{|x-y|^{p+1}} \, dx \, dy \\[6pt]
& + C \mathop{\int_0^{x_{1, k}} \int_0^{x_{1, k}}}_{|x - y | > \delta_k}  \frac{ \delta_k^p }{|x - y|^{p + 1}} \, dx \, dy. 
\end{align*}
Since $|\hat u_k(x) - \hat u_k(y)| \le C |x - y|$ for $x,y \in (0, x_{1, k})$, we obtain 
\begin{equation*}
I
\le C \mathop{\int_0^{x_{1, k}} \int_0^{x_{1, k}}}_{|x - y | \le \delta_k}  \delta_k^{-1} \, dx \, dy + C \mathop{\int_0^{x_{1, k}} \int_0^{x_{1, k}}}_{|x - y | > \delta_k}  \frac{ \delta_k^p }{|x - y|^{p + 1}} \, dx \, dy. 
\end{equation*}
It follows from straightforward integral estimates  that 
\begin{equation}\label{lem1-e1}
I \le C  x_{1, k} \mathop{\le}^{\eqref{lem1-x1k}} C c_k. 
\end{equation}

We next consider $II$.  It is clear from the definition of $\hat u_k$ that 
\begin{multline*}
II = 2 \int_{2x_{1,k}/3}^{x_{1, k}} \, dx  \int_{x_{1, k}}^{x_{2, k}} \frac{\varphi_{\delta_k}(|\hat u_k(x) - \hat u_k(y)|)}{|x - y|^{p+1}}  dy  \\[6pt]
\le  \frac{2 x_{1, k}}{3}\int_{x_{1, k}}^{x_{2, k}}   \frac{\varphi_{\delta_k}(|u_k(x_{1, k}) -  u_k(y)|)}{|x_{1, k}- y|^{p+1}}  dy, 
\end{multline*}
which implies, by \eqref{lem1-x1k} and \eqref{lem1-o2},  
\begin{equation}\label{lem1-e2}
II   \le C c_k. 
\end{equation}

Similarly, using \eqref{lem1-o3} and \eqref{lem1-o4}, one has
\begin{equation}\label{lem1-e3}
III  = \Lambda_{\delta_k}(\hat u_k, (x_{1, k},  x_{2, k}))  \le C c_k 
\end{equation}
and
\begin{equation}\label{lem1-e4}
IV  = 2  \int_{x_{2,k}}^{ (1 + 2x_{2, k})/3}  \int_{x_{1, k}}^{x_{2, k}} \frac{\varphi_{\delta_k}(|\hat u_k(x) - \hat u_k(y)|)}{|x - y|^{p+1}}  dy \, dx \le C c_k.
\end{equation}

It is clear  from  \eqref{lem1-p3} that 
\begin{equation}\label{lem1-e4-1}
V = \Lambda_{\delta_k}(\hat u_k, (x_{1, k},  x_{2, k}))   \le \Lambda_{\delta_k} (u_k, (0, 1)) \le \C +  c_k. 
\end{equation}

We now consider $VI$. We  have, for every $c > 0$,  
\begin{equation*}
\mathop{\int_0^1 \int_0^1}_{|x - y| > c} \frac{\varphi_{\delta_k}(|\hat u_k(x) - \hat u_k(y)|)}{|x - y|^{p+1}}  dy \, dx \mathop{\le}^{\eqref{cond-varphi-1}} C \delta_k^p/ c^p, 
\end{equation*}
which yields 
\begin{multline*}
VI = \mathop{\int_0^1 \int_0^1}_{|x - y| > \min\{x_{1, k}/3, (1- x_{2, k})/3 \}} \frac{\varphi_{\delta_k}(|\hat u_k(x) - \hat u_k(y)|)}{|x - y|^{p+1}}  dy \, dx \\[6pt]
\le \frac{C\delta_k^p}{\min\{x_{1, k}/3, (1- x_{2, k})/3 \}^p}. 
\end{multline*}
From \eqref{lem1-x1k} and \eqref{lem1-x2k}, we derive  that 
\begin{equation}\label{lem1-e5}
VI \le  C \delta_k^p/c_k^p \mathop{ \le }^{\eqref{lem1-p1}} C c_k. 
\end{equation}
Combining \eqref{lem1-decomp}-\eqref{lem1-e5}  yields 
\begin{equation*}
\Lambda_{\delta_k} (\hat u_k, (0, 1)) \le \C + C c_k.  
\end{equation*}
The proof of Claim \eqref{lem1-e} is complete.  In view of the definition of $\C$, we obtain \eqref{lem1-start-1} from \eqref{lem1-e}.

As a consequence of \eqref{lem1-e}, we have
\begin{equation}\label{lem1-e-***}
\limsup_{k \to + \infty} \Lambda_{\delta_k} (\hat u_k, (0, 1)) \le \limsup_{k \to + \infty} \Lambda_{\delta_k}(u_k, (0, 1)).   
\end{equation}

From $(\hat u_k)$, we now construct {\it a family} $(v_\delta) \subset L^p(0, 1)$  such that $v_\delta \to U$ in $L^p(0, 1)$ and 
\begin{equation*}
\lim_{\delta \to 0} \Lambda_{\delta} (v_\delta, (0, 1)) = \C.  
\end{equation*}

Let $(\tau_k) \subset (0,1)$ be a decreasing sequence converging to 0 such that 
\begin{equation}\label{tau}
\tau_k \le \delta_k c_k.
\end{equation}
For each $\delta \in (0, 1)$ small,  let $k = k(\delta)$ be such that $\tau_{k(\delta)+1} < \delta \le  \tau_{k(\delta)}$. Define
$ \hat m = \hat m (\delta)=\delta_{k(\delta)}/\delta$ and set $m = m (\delta)= [\hat m (\delta)]$, the
largest integer less than or equal to $\hat m (\delta)$.  Then, by \eqref{tau},  
\begin{equation*}
\hat m (\delta) \ge \delta_{k(\delta)}/ \tau_{k(\delta)} \ge 1/ c_{k(\delta)}. 
\end{equation*}
This implies 
\begin{equation}\label{lem1-m}
\hat m (\delta)/ m (\delta) \to 1 \mbox{ as } \delta \to 0. 
\end{equation}
In what follows, for notational ease, we delete the dependence on $\delta$ in $k(\delta)$, $m (\delta)$ and $\hat m (\delta)$. 

Consider $v_\delta: (0,1) \to
\mR$ defined as follows
\begin{equation*}
v_\delta(x) : = \frac{1}{\hat m} \hat v_\delta (m x),
\end{equation*}
where $\hat v_\delta: (0, m) \to \mR$ is given by, for $\tau_{k+1} < \delta \le  \tau_{k}$,  
\begin{equation}\label{lem1-def-hatv}
\hat v_\delta (x) : = [x]  + \hat u_k (x - [x]). 
\end{equation}
Then
\begin{equation}\label{lem1-est1*}
\Lambda_\delta(v_\delta, (0, 1)) = \frac{m^{p-1}}{\hat m^p} \Lambda_{\delta_k} (\hat v_\delta, (0, m)). 
\end{equation}
We have
\begin{multline}\label{lem1-decomposition}
\Lambda_{\delta_k} (\hat v_\delta, (0, m))  =  \sum_{i=0}^{m-1} 
\Lambda_{\delta_k} (\hat v_\delta, (i, i+1)) 
\\[6pt]+ \sum_{i=0}^{m-1} \int_{i}^{i+1} \, dx \mathop{\int}_{(0, m) \setminus (i, i+1)} \frac{\varphi_{\delta_k}(|\hat v_\delta(x) - \hat v_\delta(y)|)}{|x-y|^{p+1}}  \, dy. 
\end{multline}

It is clear from the definition of $v_\delta$ that, for $0 \le i \le m-1 $,  
\begin{equation}\label{lem1-est1*-1}
\Lambda_{\delta_k} (\hat v_\delta, (i, i+1)) = \Lambda_{\delta_k} (\hat u_k, (0, 1))
\end{equation}
and, that, with $s_k = \min\{x_{1,k},  1- x_{2, k} \}/3$,  
\begin{multline}\label{lem1-est1*-2-1}
\int_{i}^{i+1} \, dx \mathop{\int}_{(0, m) \setminus (i, i+1)} \frac{\varphi_{\delta_k}(|\hat v_\delta(x) - \hat v_\delta(y)|)}{|x-y|^{p+1}}  \, dy  \\[6pt]
=  \int_{i}^{i+1} \, dx \mathop{\mathop{\int}_{(0, m) \setminus (i, i+1)}}_{|y-x| \le s_k} \frac{\varphi_{\delta_k}(|\hat v_\delta(x) - \hat v_\delta(y)|)}{|x-y|^{p+1}}  \, dy \\[6pt]
+ \int_{i}^{i+1} \, dx \mathop{\mathop{\int}_{(0, m) \setminus (i, i+1)}}_{|y-x| \ge s_k} \frac{\varphi_{\delta_k}(|\hat v_\delta(x) - \hat v_\delta(y)|)}{|x-y|^{p+1}}  \, dy.  
\end{multline}
Note that, by \eqref{lem1-x1k} and \eqref{lem1-x2k},  
$$
s_k \sim c_k. 
$$
By the same method used to establish \eqref{lem1-e5}, we have 
\begin{equation}\label{lem1-est1*-2-2}
 \int_{i}^{i+1} \, dx \mathop{\mathop{\int}_{(0, m) \setminus (i, i+1)}}_{|y-x| \ge s_k} \frac{\varphi_{\delta_k}(|\hat v_\delta(x) - \hat v_\delta(y)|)}{|x-y|^{p+1}}  \, dy \le C c_k. 
\end{equation}
Note,  from the definition of $\hat v_\delta$,  that $|\hat v_\delta(x) - \hat v_{\delta} (y)| = |x -y|$ for $x \in (i, i+1)$ and $y \in (0, m) \setminus (i, i+1)$  with $|y-x| \le s_k$.  By the same method used to establish \eqref{lem1-e1}, we have 
\begin{equation}\label{lem1-est1*-2-3}
 \int_{i}^{i+1} \, dx \mathop{\mathop{\int}_{(0, m) \setminus (i, i+1)}}_{|y-x| \le s_k} \frac{\varphi_{\delta_k}(|\hat v_\delta(x) - \hat v_\delta(y)|)}{|x-y|^{p+1}}  \, dy \le   C c_k. 
\end{equation}
We derive from  \eqref{lem1-est1*-2-1}-\eqref{lem1-est1*-2-3} that 
\begin{equation}\label{lem1-est1*-2}
\int_{i}^{i+1} \, dx \mathop{\int}_{(0, m) \setminus (i, i+1)} \frac{\varphi_{\delta_k}(|\hat v_\delta(x) - \hat v_\delta(y)|)}{|x-y|^{p+1}}  \, dy  \le C c_k. 
\end{equation}
Combining \eqref{lem1-est1*}-\eqref{lem1-est1*-2} and  using \eqref{lem1-p1} and \eqref{lem1-e}, we have
\begin{equation}\label{lem1-est2}
\Lambda_{\delta_k} (\hat v_\delta, (0, m))  \le  m (\C + C c_k). 
\end{equation}
We deduce from \eqref{lem1-m}, \eqref{lem1-est1*}, and \eqref{lem1-est2} that 
\begin{equation}\label{lem1-cl1}
\limsup_{\delta \to 0} \Lambda_{\delta} (v_\delta, (0, 1)) \le \C. 
\end{equation}

It is clear from \eqref{lem1-m} that 
\begin{equation*}
\lim_{\delta \to 0} \int_{0}^1 |v_\delta - U|^p \, dx = \lim_{\delta \to 0} \frac{1}{m^p} \int_{0}^1 |\hat v_\delta (mx) - mx|^p \, dx, 
\end{equation*} 
which yields, by a change of variables, 
\begin{equation*}
\lim_{\delta \to 0} \int_{0}^1 |v_\delta - U|^p \, dx = \lim_{\delta \to 0} \frac{1}{m^{p+1}} \int_{0}^m |\hat v_\delta (x) - x|^p \, dx. 
\end{equation*} 
Since 
$$
\int_{0}^m |\hat v_\delta (x) - x|^p \, dx = \sum_{j=0}^{m-1} \int_{j}^{j+1} |\hat v_\delta (x) - x|^p \, dx, 
$$
it follows from \eqref{lem1-p2} and \eqref{lem1-def-hatv} that 
\begin{equation}\label{lem1-cl2}
\lim_{\delta \to 0} \int_{0}^1 |v_\delta - U|^p \, dx = 0. 
\end{equation} 

Combining \eqref{lem1-cl1} and \eqref{lem1-cl2}, and using the definition of $\C$, we obtain \eqref{lem1-statement}. The proof is complete. 
\end{proof}

We next establish

\begin{lemma}\label{lem2} Let $p \ge 1$ and let $\varphi$ satisfy \eqref{cond-varphi-0}-\eqref{cond-varphi-3}. Let $a< b$ and let  $u$ be an  affine function on $(a, b)$. There exists a family $(u_\delta) \subset L^p(a, b)$ such that $u_\delta \to u$ in $L^p(a, b)$,  as $\delta \to 0$, 
\begin{equation*}
\limsup_{\delta \to 0} \Lambda_{\delta} (u_\delta, (a, b)) \le \C \int_a^b |u'|^p \, dx, 
\end{equation*}
and, for small $\delta$, 
\begin{equation*}
u_{\delta} = u  \mbox{ on } (a, a + \delta^{1/2}/6) \cup (b - \delta^{1/2}/ 6, b). 
\end{equation*}
\end{lemma}

\begin{proof} By Lemma~\ref{lem1}, after a change of variables, there exists a family $(v_\delta) \subset L^p(a,b)$ converging to $u$ in $L^p(a, b)$ as $\delta \to 0$, and 
\begin{equation}\label{lem2-p1}
\lim_{\delta \to 0} \Lambda_{\delta} (v_\delta, (a, b)) = \C \int_a^b |u'|^p \, dx. 
\end{equation}
As in the proof of Lemma~\ref{lem1}, there exist $(c_\delta) $, $(x_{1, \delta})$, $(x_{2, \delta})$ such that 
\begin{equation*}
\lim_{\delta \to 0} c_\delta = 0, \quad c_\delta \ge \delta^{1/2}, 
\end{equation*}
\begin{equation*}
\Lambda_{\delta}(v_\delta, (a, b)) \le \C \int_a^b |u'|^p \, dx + c_{\delta}, 
\end{equation*}
\begin{equation*}
x_{1, \delta} \in (a + c_\delta/ 2, a + c_\delta ), \quad x_{2, \delta} \in (b - c_\delta, b - c_\delta/ 2), 
\end{equation*}
\begin{equation*}
|v_{\delta}(x_{1, \delta}) - u(x_{1 \delta})| \le c_{\delta}, \quad |v_{\delta}(x_{2, \delta}) - u(x_{2, \delta})| \le c_{\delta}, 
\end{equation*}
\begin{equation*}
\int_a^b \frac{\varphi_{\delta}(|v_\delta(x_{1,\delta}) - v_\delta(y)|)}{|x_{1, \delta} - y|^{p+1}}  dy  \le C, \quad  \int_a^b \frac{\varphi_{\delta}(|v_\delta(x_{2,\delta}) - v_\delta(y)|)}{|x_{2, \delta} - y|^{p+1}}  dy  \le C, 
\end{equation*} 
for small $\delta$ and  for some positive constant $C$ independent of $\delta$. Here we used the fact, for $c \in (0, (b-a)/2)$, 
$$
\liminf_{\delta \to 0} \Lambda_{\delta} (v_\delta, (a+c, b-c)) \ge \C \int_{a+c}^{b-c} |u'|^p \, dx.
$$

Define $\hat v_\delta: (a, b)  \to \mR$ as follows
\begin{equation*}
\hat v_\delta (x) : = \left\{\begin{array}{cl}
u(x)& \mbox{ if } a <   x <  \frac{2 a + x_{1, \delta}}{3}, \\[6pt]
v_\delta(x_{1, \delta}) & \mbox{ if } \frac{a + 2x_{1, \delta}}{3} <   x < x_{1, \delta}, \\[6pt]
v_\delta(x) & \mbox{ if } x_{1, \delta} \le x \le x_{2, \delta}, \\[6pt]
v_\delta(x_{2, \delta}) & \mbox{ if } x_{2, \delta} < x <  \frac{b+2x_{2, \delta}}{3}, \\[6pt]
u(x) & \mbox{ if } \frac{2b+x_{2, \delta}}{3} < x <  b, 
\end{array}\right.
\end{equation*}
and $\hat v_\delta$ is chosen in $[\frac{2 a + x_{1,\delta}}{3}, \frac{a + 2 x_{1, \delta}}{3}] \cup [ \frac{b+2x_{2, \delta}}{3}, \frac{2 b+x_{2, \delta}}{3}]$ in such a way that it is affine there and $\hat u_\delta$ is  continuous at the end points. It is clear that $\hat v_\delta \to u$ in $L^p(a, b)$.

As in  the proof of \eqref{lem1-e-***}, we have 
\begin{equation*}
\limsup_{\delta \to 0} \Lambda_{\delta}(\hat v_\delta, (a, b)) \le \limsup_{\delta \to 0} \Lambda_{\delta}(v_\delta, (a, b)).
\end{equation*}
By \eqref{lem2-p1}, the conclusion now holds for $(u_\delta)$  with $u_\delta : = \hat v_\delta$. 
\end{proof}

Using Lemma~\ref{lem2}, we can establish the following key ingredient in the proof of (G2). 

\begin{lemma}\label{lemaffine1} Let $p \ge 1$ and let $\varphi$ satisfy \eqref{cond-varphi-0}-\eqref{cond-varphi-3}. 
Let $u$ be a continuous piecewise linear function defined on $\mR$ with compact
support. There exists a family $(u_\delta) \subset L^p(\mR)$ such that $u_\delta \to u$  in $L^p(\mR)$, as $\delta \to 0$,  and
\begin{equation*}
\limsup_{\delta \to 0} \Lambda_\delta(u_\delta, \mR) \le \C \int_{\mR}
|u'|^p \, dx.
\end{equation*} 
\end{lemma}

\begin{proof} Since $u$ is a  continuous piecewise linear
function defined on $\mR$ with compact support, there exist $a_1 < a_2 < \dots <
a_m$ such that $u$ is affine on $(a_i, a_{i+1})$, $1 \le i < m-1$,
$u(x) = 0 $ if $x< a_1$ or $x > a_m$, and $u$ is continuous at $a_i$ for $1 \le i \le m$. In what follows, we denote $a_0 = - \infty$ and $a_{m+1}  = + \infty$. 

For $1 \le i \le m-1$, by Lemma~\ref{lem2},  there exist 
a family $(v_{i, \delta}) \subset L^p(a_i, a_{i+1})$ such that 
\begin{equation}\label{lem3-vi-1-1}
v_{i, \delta} \to u \mbox{ in } L^{p}(a_{i}, a_{i+1}) \mbox{ as } \delta \to 0, 
\end{equation}
\begin{equation}\label{lem3-vi-A-1}
\limsup_{\delta \to 0} \Lambda_{\delta} ( v_{i, \delta}, (a_i, a_{i+1}) )  \le \C \int_{a_i}^{a_{i+1}} |u'|^p \, dx, 
\end{equation}
and, for small $\delta$, 
\begin{equation}\label{lem3-vi-2}
v_{i, \delta}  =  u  \mbox{ on } (a_i, a_i + \delta^{1/2}/6) \cup (a_{i+1} -  \delta^{1/2}/6, a_{i+1}). 
\end{equation}
Set
\begin{equation}\label{lem3-vi-1-2}
v_{0, \delta} = 0 \mbox{ in } (a_0, a_1) \quad { and  } \quad v_{m, \delta} = 0 \mbox{ in } (a_m, a_{m+1}). 
\end{equation}
Then 
\begin{equation}\label{lem3-vi-A-2}
 \Lambda_{\delta} ( v_{0, \delta}, (a_0, a_{1}) )   =   \Lambda_{\delta} ( v_{m, \delta}, (a_m, a_{m+1}) )  = 0. 
\end{equation}
Define $u_\delta : \mR \to \mR$ as follows 
\begin{equation}\label{lem3-def-u}
u_\delta (x) = v_{i, \delta} (x) \mbox{ for } x \in  (a_i, a_{i+1}) \mbox{ and }  0 \le i \le m.  
\end{equation}
As in \eqref{lem1-est2}, we have
\begin{equation}\label{lem3-vi-***}
\Lambda_{\delta}(u_\delta, \mR)  \le  \sum_{i=1}^{m-1}  \Lambda_{\delta} (u_\delta, (a_i, a_{i+1})) +  C m \delta^{1/2},  
\end{equation}
for some positive constant $C$ independent of $\delta$ (but $C$ depends on the slope of $u$ on each interval $(a_i, a_{i+1})$ for small $\delta$).  It follows from \eqref{lem3-vi-A-1}, \eqref{lem3-vi-A-2}, and \eqref{lem3-vi-***} that 
\begin{equation}\label{lem3-est-u-1}
\limsup_{\delta \to 0} \Lambda_{\delta}(u_\delta, \mR)  \le   \C \int_{\mR} |u'|^p \, dx. 
\end{equation}
From \eqref{lem3-vi-1-1} and \eqref{lem3-vi-1-2}, we have 
\begin{equation}\label{lem3-est-u-2}
u_\delta \to u \mbox{ in } L^p(\mR) \mbox{ as } \delta \to 0. 
\end{equation}
The conclusion now follows from \eqref{lem3-est-u-1} and \eqref{lem3-est-u-2}. 
\end{proof}

We are ready to complete the 

\begin{proof}[Proof of Property (G2)] We distinguish two cases. 

{\bf Case 1}: $\C > 0$. In this case, for any function $g \in W^{1, p}(\mR)$ with $p>1$ (resp. $g \in BV(\mR)$ with $p=1$), we  will construct a family $(g_\delta) \subset L^p(\mR)$ such that $g_\delta \to g$ in $L^p(\mR)$, as $\delta \to 0$, and 
$$
\limsup_{\delta \to 0} \Lambda_{\delta} (g_\delta, \mR) \le \C \int_{\mR} |g'|^p \, dx. 
$$ 

{\bf Case 2}: $\C = 0$. In this case, for any function $g \in L^p(\mR)$ with $p \ge 1$, 
we will  construct a family $(g_\delta) \subset L^p(\mR)$ such that $g_\delta \to g$ in $L^p(\mR)$, as $\delta \to 0$, and 
$$
\lim_{\delta \to 0} \Lambda_{\delta} (g_\delta, \mR) = 0. 
$$ 

\noindent {\bf Proof in Case 1}:  Let $(g_n) \subset L^p(\mR)$ be a sequence of continuous piecewise linear functions with compact support such that $g_n \to g$ in $L^p(\mR)$, as $n \to + \infty$, and 
\begin{equation*}
\lim_{n \to + \infty} \int_{\mR} |g_n \sp{\prime} |^p \, dx = \int_{\mR} |g'|^p \, dx. 
\end{equation*}
For each $n \in \mN$, by Lemma~\ref{lemaffine1}, there exists a family 
$(g_{n, \delta}) \subset L^p(\mR)$ such that $g_{n, \delta} \to g_n$ in $L^p(\mR)$, as $\delta \to 0$,  and 
\begin{equation*}
\limsup_{\delta \to 0 } \Lambda_{\delta}(g_{n, \delta}, \mR) \le \C \int_{\mR} |g_n \sp{\prime}|^p \, dx. 
\end{equation*}
The conclusion now follows from a standard selection process. 

\medskip 

\noindent  {\bf Proof in Case 2}:  Let $(g_n) \subset L^p(\mR)$ be a sequence of continuous piecewise linear functions with compact support such that $g_n \to g$ in $L^p(\mR)$ as $n \to \infty$. 
For each $n \in \mN$, by Lemma~\ref{lemaffine1}, there exists a family 
$(g_{n, \delta}) \subset L^p(\mR)$ such that $g_{n, \delta} \to g_n$ in $L^p(\mR)$, as $\delta \to 0$,  and 
\begin{equation*}
\lim_{\delta \to 0 } \Lambda_{\delta}(g_{n, \delta}, \mR) = 0. 
\end{equation*}
The conclusion now follows from a standard selection process. 
\end{proof}

\subsection{Proof of Property (G1)}

This section containing two subsections is devoted to the proof of Property (G1). In the first subsection, we consider the case $p=1$. The case $p>1$ is studied in the second subsection. 

\subsubsection{Proof of Property (G1) for $p=1$} \label{sect-G1-1}

In this section, we consider   $p=1$ and assume $\C > 0$ since there is nothing to prove otherwise.  
Define
\begin{equation}\label{defb1}
\gamma := \inf \liminf_{\delta \to 0} \Lambda_{\delta} (v_\delta, (0, 1)), 
\end{equation}
where the infimum is taken over all families
$(v_\delta) \subset L^1(0, 1)$ such that $v_\delta  \to H_{1/2}$ in $L^1(0, 1)$ as $\delta \to 0$. Here and
in what follows $H_c(x)  := H(x- c)$ for any $c \in \mR$, where
$H$ is the function defined on $\mR$ by
\begin{equation*}
H(x) : = \left\{\begin{array}{ll} 0 & \mbox{if } x < 0, \\[6pt]
1 & \mbox{otherwise}.
\end{array}\right.
\end{equation*}

There are two key ingredients. 
\begin{lemma}\label{lem-G1} Let $p = 1$ and let $\varphi$ satisfy \eqref{cond-varphi-0}-\eqref{cond-varphi-3}. We have
$$
\gamma = \C, 
$$
where $\C$ is the constant defined in \eqref{def-k}. 
\end{lemma}

\begin{lemma}\label{lem-C4} Let $p = 1$ and let $\varphi$ satisfy \eqref{cond-varphi-0}-\eqref{cond-varphi-3}.
Let $u \in L^1( a, b)$ and let $a <  t_1 < t_{2} <  b$ be two Lebesgue points of $u$. Let $(u_\delta)  \subset L^1(t_1, t_2)$ 
such that $u_\delta \to u$ in $L^{1}(t_1, t_2)$. We have
\begin{equation}\label{lem-C4-state}
\liminf_{\delta \to 0} \Lambda_{\delta}(u_{\delta}, (t_1, t_2)) 
\ge \gamma |u(t_2) - u(t_1)|.
\end{equation}
\end{lemma}

Assuming \Cref{lem-G1,lem-C4}, we give the

\begin{proof}[Proof of Property (G1) for $p=1$] Since $\gamma = \C$ by \Cref{lem-G1}, Property (G1) is now a direct consequence of \Cref{lem-C4} and the fact that for $u \in L^1(\mR)$, then 
\begin{equation}\label{G1-pro}
\int_{\mR} |u'| \, dx =  \sup \left\{  \sum_{j=1}^m |u(t_{j+1}) - u(t_j)|\right\},  
\end{equation}
 where the supremum is taken over all finite sets $\big\{t_j;  1 \le j \le m+1 \big\}$ such that $t_1 <  \cdots < t_{m+1}$ and each $t_j$ is a Lebesgue point of $u$, see,  e.g., \cite[Theorem 1 on page 217]{EGMeasure}. Indeed, we have 
 \begin{equation*}
 \liminf_{\delta \to 0} \Lambda_{\delta} (u_\delta, \mR) \ge \sum_{j=1}^{m} \liminf_{\delta \to 0} \Lambda_{\delta} (u_\delta, (t_j, t_{j+1}))  \ge \sum_{j=1}^{m} \gamma |u(t_{j+1}) - u(t_j)| \mbox{ by Lemma~\ref{lem-C4}}, 
 \end{equation*}
which implies, by \eqref{G1-pro}, 
 \begin{equation*}
 \liminf_{\delta \to 0} \Lambda_{\delta} (u_\delta, \mR) \ge  \gamma \int_{\mR} |u'| \, dx.  
 \end{equation*}
\end{proof}

The proof of \Cref{lem-G1} relies on the two lemmas below. The first one is
 
\begin{lemma}\label{lem-C1} Let $p = 1$ and let $\varphi$ satisfy \eqref{cond-varphi-0}-\eqref{cond-varphi-3}.  There exist a sequence $(h_k) \subset L^1(0, 1)$ and a sequence
$(\delta_k)\subset \mR_+$ converging to 0  such that 
$$
\lim_{k \to + \infty} h_k = H_{1/2} \mbox{ in } L^1(0, 1), 
$$
$$
h_k(x) = 0 \mbox{ for  } x  < 1/16, \quad h_k(x) = 1 \mbox{ for } x > 1-1/16, 
$$
and
\begin{equation*}
\limsup_{k \to +  \infty} \Lambda_{\delta_k} (h_k, (0, 1))
\le \gamma.
\end{equation*}
\end{lemma}

\begin{proof}  Let $(\delta_k) \subset \mR_+$ and $(g_k) \subset L^1(0, 1)$ be such that 
$$
\lim_{k \to + \infty} \delta_k =0, \quad \lim_{k \to + \infty} g_k  = H_{1/2} \mbox{ in } L^1(0, 1), 
$$
and 
$$
\lim_{k \to + \infty} \Lambda_{\delta_k} (g_k, (0, 1)) = \bb. 
$$
Let $(c_k) \subset \mR_+$ be such that, for large $k$,  
\begin{equation}\label{lem-C1-ob-1}
\lim_{k \to + \infty} c_k = 0, \quad c_k \ge \delta_k^{1/2}, 
\end{equation}
\begin{equation}\label{lem-C1-ob-2}
\Lambda_{\delta_k} (g_k, (0, 1)) \le \gamma + c_k, 
\end{equation}
\begin{equation}\label{lem-C1-ob-3}
\int_{0}^1 |g_k - H_{1/2}| \, dx  \le c_k^{2},  
\end{equation}

In what follows in this proof, $C$ denotes a positive constant depending only on $\alpha$ and $\beta$. 
From \eqref{lem-C1-ob-1}-\eqref{lem-C1-ob-3}, we derive that, for some $\tau_k \in (1/8, 1/5)$ and  with $\hat \tau_k = \tau_k + 1/ 2$,  
\begin{equation}\label{lem-C1-ob-4}
\int_{\tau_k}^{\tau_k + c_k} |g_k - H_{1/2}| \, dx  + \int_{\hat \tau_k}^{\hat \tau_k + c_k} |g_k - H_{1/2}| \, dx  \le C c_k
\end{equation}
and 
\begin{multline}\label{lem-C1-ob-5}
\mathop{\iint}_{(\tau_k, \tau_k  + c_k) \times (0, 1) } \frac{\varphi_{\delta_k} (|g_k(x) - g_k(y)|)}{|x-y|^{2}} \, dx \, dy \\[6pt]
+ \mathop{\iint}_{(\hat \tau_{k}, \hat \tau_k  + c_k) \times (0, 1) } \frac{\varphi_{\delta_k} (|g_k(x) - g_k(y)|)}{|x-y|^{2}} \, dx \, dy \le C c_k.     
\end{multline}
It follows from \eqref{lem-C1-ob-4} and \eqref{lem-C1-ob-5} that, for some $b_k \in [1/8, 1/4]$,  with $\hb_k = b_k + 1/2$,  
\begin{equation}\label{lem-C1-p1}
|g_k(b_k)| + |g_k(\hb_k) - 1| \le C c_k
\end{equation}
and
\begin{equation}\label{lem-C1-p2}
\int_0^1 \frac{\varphi_{\delta_k}(|g_k(b_k) - g_k(y)|)}{|b_k-y|^{2}}  \, dy + \int_0^1 \frac{\varphi_{\delta_k}(|g_k(\hb_k) - g_k(y)|)}{|\hb_k-y|^{2}}  \, dy\le C. 
\end{equation}


Define $h_k: (0, 1) \to \mR$ as follows
\begin{equation*}
h_k (x) = \left\{\begin{array}{cl}
0 & \mbox{ for } 0 <  x <  b_k - 2 c_k,   \\[6pt]
g_k(b_k) & \mbox{ for  }  b_k - c_k  <  x  <  b_k, \\[6pt]
g_k(x) & \mbox{ for  }  b_k \le x  \le \hb_k, \\[6pt]
g_k(\hb_k) & \mbox{ for  }  \hb_k  < x  <  \hb_k + c_k, \\[6pt]
1 & \mbox{ for } \hb_k + 2 c_k <  x <  1,   
\end{array}\right.
\end{equation*}  
and $h_k$ is chosen in $[b_k - 2c_k,  b_k - c_k] \cup [\hb_k + c_k, \hb_k + 2 c_k]$ in such a way that it is affine there and $h_k$ is  continuous at the end points.  As in the proof of \eqref{lem1-e} in Lemma~\ref{lem1}, one can check that 
\begin{equation}\label{lem-C1-p6}
\Lambda_{\delta_k}(h_k, (0, 1)) \le \gamma + C c_k. 
\end{equation}
Therefore, the conclusion holds for $h_k$.  
\end{proof} 

The second lemma used in the proof of \Cref{lem-G1} is 

\begin{lemma}\label{lem-C2} Let $p = 1$ and let $\varphi$ satisfy \eqref{cond-varphi-0}-\eqref{cond-varphi-3}. There exist a sequence $(u_k) \subset L^1(0, 1)$ and a sequence
$(\mu_k) \subset \mR_+$  such that 
$$
\lim_{k \to + \infty} \mu_k = 0, \quad \lim_{k \to + \infty} u_k = U  \mbox{ in } L^1(0, 1), 
$$
and
\begin{equation*}
\limsup_{k \to + \infty} \Lambda_{\mu_k} (u_k, (0, 1))
\le \gamma.
\end{equation*}
\end{lemma}

\begin{proof} Let $(\delta_k)$ and  $(h_k) \subset L^1(0, 1)$ be the sequences satisfying the conclusion of Lemma~\ref{lem-C1}.  Given $n \in \mN$, set $I_j=  (j/n, (j+1)/ n) $ for 
$0 \le j \le n-1$,  and define
\begin{equation}\label{def-fkn}
f_{k, n} (x) =  \frac{1}{n} h_k \Big( n(x - j/n) \Big) + \frac{j}{n}  \mbox{ for } x \in I_j. 
\end{equation}
By a change of variables,   we obtain 
\begin{equation}\label{lem-C2-p1}
\int_{0}^1 |f_{k, n} (x) - x| \, dx  =  \frac{1}{n} \int_{0}^1 |h_k (x) - x| \, dx.  
\end{equation}

We next estimate
\begin{equation*}
 \Lambda_{\delta_k/n} (f_{k, n}, (0, 1)).
\end{equation*}
It is clear that 
\begin{multline}\label{lem-C2-p2}
\Lambda_{\delta_k/n} (f_{k, n}, (0, 1))
\le   \sum_{j=0}^{n-1} \Lambda_{\delta_k/n} (f_{k, n}, I_j) \\[6pt]
 +  \sum_{j=0}^{n-1} \mathop{\iint}_{I_j \times ( (0, 1) \setminus I_j)}
 \frac{ \varphi_{\delta_k/n} (|f_{k, n}(x) - f_{k,n}(y)|)}{|x-y|^{2}} \, dx \, dy.
\end{multline}
We have, by a change of variables,  
\begin{equation}\label{lem-C2-p3}
\Lambda_{\delta_k/n} (f_{k, n}, I_j) = \frac{1}{n}\Lambda_{\delta_k} (h_k, (0, 1)). 
\end{equation}
It is clear that 
\begin{multline}
\mathop{\iint}_{I_j \times ((0, 1) \setminus I_j)}
 \frac{ \varphi_{\delta_k/n} (|f_{k, n}(x) - f_{k,n}(y)|)}{|x-y|^{2}} \, dx \, dy \\[6pt]
 = \mathop{ \mathop{\iint}_{I_j \times ((0, 1) \setminus I_j)}}_{|x - y| < 1/(16 n)}
 \frac{ \varphi_{\delta_k/n} (|f_{k, n}(x) - f_{k,n}(y)|)}{|x-y|^{2}} \, dx \, dy \\[6pt] + \mathop{ \mathop{\iint}_{I_j \times ((0, 1) \setminus I_j)}}_{|x -y| > 1/(16n)}
 \frac{ \varphi_{\delta_k/n} (|f_{k, n}(x) - f_{k,n}(y)|)}{|x-y|^{2}} \, dx \, dy. 
\end{multline}
Since the first term on the RHS of the above identity is 0,  by straightforward integral estimates, we obtain 
\begin{equation}\label{lem-C2-p4}
\mathop{\iint}_{I_j \times ((0, 1) \setminus I_j)}
 \frac{ \varphi_{\delta_k/n} (|f_{k, n}(x) - f_{k,n}(y)|)}{|x-y|^{2}} \, dx \, dy \le C \frac{\delta_k}{n} \ln n. 
\end{equation}
Set 
$$
\mbox{$n_k=  [\ln \delta_k^{-1}]$ (the integer part of $\ln \delta_k^{-1}$)} \quad \mbox{ and } \quad \mu_k = \delta_k/ n_k,  
$$
so that $n_k \to + \infty$ and $\mu_k \to 0$ as $k \to + \infty$. 
Combining \eqref{lem-C2-p2}, \eqref{lem-C2-p3},  and \eqref{lem-C2-p4} yields, with $u_k  = f_{k, n_k}$,   
\begin{equation}\label{lem-C2-p5}
 \Lambda_{\mu_k} (u_k, (0, 1)) \le \Lambda_{\delta_k} (h_k, (0, 1)) + C \delta_k \ln n_k. 
\end{equation}
It follows that 
\begin{equation}\label{lem-C2-p6}
\limsup_{k \to + \infty} \Lambda_{\mu_k} (u_k, (0, 1)) \le  \gamma 
\end{equation}
since $\limsup_{k \to + \infty} \Lambda_{\delta_k}(h_k, (0, 1)) \le \gamma$ by Lemma~\ref{lem-C1}. 
Since $n_k \to + \infty$, we have, by \eqref{lem-C2-p1}, 
\begin{equation}\label{lem-C2-p7}
\lim_{k \to + \infty} \int_{0}^1 |u_{k} - U| \, dx = 0. 
\end{equation}
The conclusion follows from \eqref{lem-C2-p6} and \eqref{lem-C2-p7}. 
\end{proof}

We are ready to give the 
\begin{proof}[Proof of \Cref{lem-G1}] By Property (G2) applied with $g = H_{1/2}$ and $I =(0, 1)$, there exists a family $(g_\delta) \subset L^1(0, 1)$ such that $g_\delta \to H_{1/2}$ in $L^1(0, 1)$ and
$$
\limsup_{\delta \to 0} \Lambda_\delta(g_\delta, (0, 1)) \le \C.  
$$
This implies, by the definition of $\gamma$ in  \eqref{defb1}, that $\gamma \le \C$.  By \Cref{lem-C2} and Remark~\ref{rem-k}, one obtains $\C \le \gamma$. The conclusion follows. 
\end{proof}

We now give the 

\begin{proof}[Proof of \Cref{lem-C4}]
We begin with the following 

\medskip 
{\it Claim 1:} For any $\eps > 0$, there exist two positive numbers
$\hat \delta_1$,  $\hat \delta_2$ such that for  any $a_1, b_1, c \in \mR$ with $a_1<b_1$, and for any 
$u \in
L^1(a_1, b_1)$ satisfying   
\begin{equation*}
\|u - c H_{a_1 + \frac{1}{2} (b_1-a_1)}\|_{L^{1}(a_1, b_1)} \le  |c| (b_1 - a_1) \hat \delta_{1}, 
\end{equation*}
one has 
\begin{equation*}
\Lambda_{\delta}(u, (a_1, b_1)) \ge  |c| (\gamma - \eps) \mbox{ for all } \delta  \in (0,  |c| \hat \delta_2). 
\end{equation*}
To establish the claim, we first consider the case $(a_1, \, b_1) = (0, 1)$ and $c = 1$. The existence of $\hat \delta_{1}$ and $\hat \delta_{2}$ in this case is a direct consequence of the definition of $\gamma$ by a contradiction argument. The general case follows from this case by a change of variables. 

\medskip 
We now prove \eqref{lem-C4-state}. Without loss of generality, one may assume that $t_1= 0$, $t_2 = 1$,  $u(t_1) = 0$,  and $u(t_2) = 1$. 
It suffices to prove that 
\begin{equation}\label{lem-C4-main}
 \liminf_{k \to + \infty}  \Lambda_{\delta_k}(u_{k}, (0, 1)) \ge \gamma, 
\end{equation}
for every $(\delta_k) \subset \mR_+$ converging to 0,  and  for every $(u_k) \subset L^1(0, 1)$ such that $u_k \to u$ in $L^1(0, 1)$ and $\sup_{k} \Lambda_{\delta_k}(u_{k}, (0, 1)) < + \infty$. 

Set
\begin{equation*}
T = \sup_{k}  \Lambda_{\delta_k}(u_{k}, (0, 1)) < + \infty. 
\end{equation*}
Fix $\eps>0$ (arbitrary). Let $\hat \delta_{1} $ be the constant in the Claim corresponding to $\eps$. Without loss of generality, one may assume that $\hat \delta_1 < 1$. Let $c$ be a small positive number such that 
$$
\int_0^{c} |u(x) | \, dx  + \int_{1-c}^{1} |u(x) - 1| \, dx \le c \hat \delta_1^2/64. 
$$
Since $0$ and $1$ are Lebesgue points, such a $c$ exists.  Since $u_k \to u$ in $L^1(0, 1)$, it follows that, for large $k$, 
$$
\int_0^{c} |u_k(x)| \, dx  + \int_{1-c}^{1} |u_k(x) - 1| \, dx \le c \hat \delta_1^2/32. 
$$
This implies, for large $k$, 
$$
| A_k | \ge c/2, 
$$
where $A_k = \{ x  \in (0, c); |u_k(x)| \le \delta_1^2/16 \} $. There exists $x_{1, k} \in A_k$ such that 
\begin{equation}\label{lem-C4-p1}
 \int_{0}^{1} \frac{\varphi_{\delta_k} ( |u_{k} (x_{1, k}) - u_{k} (y) |)}{|x_{1, k} - y|^{2}} \, dy 
\le \frac{1}{|A_k|} \int_{A_k} \, dx  \int_{0}^{1} \frac{\varphi_{\delta_k}(|u_{k}(x) - u_{k}(y)| )}{|x - y|^{2}} \, dy  \le 2T /c. 
\end{equation}
Similarly, there exists $x_{2, k} \in \{ x  \in (1-c, 1); |u_k(x)-1| \le \delta_1^2/16 \}$ such that 
\begin{equation}\label{lem-C4-p2}
\int_{0}^{1} \frac{\varphi_{\delta_k} ( |u_{k} (x_{2, k}) - u_{k} (y) |)}{|x_{2, k} - y|^{2}} \, dy 
\le 2T/c. 
\end{equation}
It is then clear that 
\begin{equation}\label{lem-C4-p3}
|u_{k}(x_{1, k})| +   |u_{k}(x_{2, k}) - 1 | \le  \hat \delta_{1}^{2}/8,  
\end{equation}
for large $k$. 


\medskip
For  each (fixed) $n>0$ (large), define $v_{k} :  (-n, n) \to \mR$ as follows
\begin{equation*}
v_{k} (x) = \left\{\begin{array}{cl} u_{k} (x_{1, k}) & \mbox{ if } -n <  x <  x_{1, k}, \\[6pt]  
u_{k}(x) & \mbox{ if }  x_{1, k} \le x \le x_{2, k}, \\[6pt]
u_{k}(x_{2, k}) & \mbox{ if } x_{2, k} <  x <  n. 
\end{array} \right.
\end{equation*}
We have, since $v_k$ is constant on $(-n, x_{1, k})$ and on $(x_{2, k}, n)$, 
\begin{align}\label{lem-C4-d1}
\Lambda_{\delta_k}(v_{k}, (-n, n) ) \le & \Lambda_{\delta_k}(v_{k}, (x_{1, k}, x_{2, k}) ) \nonumber \\[6pt]
& + 2 \int_{x_{1, k} - \delta_k}^{x_{1, k}} \int_{x_{1, k}}^{x_{2, k}} \frac{\varphi_{\delta_k}(|v_{k}(x) - v_{k}(y)| )}{|x - y|^{2}} \, dy \, dx\nonumber \\[6pt] 
& + 2 \int_{x_{1, k}}^{x_{2, k}} \int_{x_{2, k}}^{x_{2, k} + \delta_k} \frac{\varphi_{\delta_k}(|v_{k}(x) - v_{k}(y)| )}{|x - y|^{2}} \, dy \, dx \nonumber \\[6pt] 
& +  2 \int_{-n}^{x_{1, k} - \delta_k} \int_{x_{1, k}}^{x_{2, k}} \frac{\varphi_{\delta_k}(|v_{k}(x) - v_{k}(y)| )}{|x - y|^{2}} \, dy \, dx\nonumber \\[6pt]
& + 2 \int_{x_{1, k}}^{x_{2, k}} \int_{x_{2, k} + \delta_k}^{n} \frac{\varphi_{\delta_k}(|v_{k}(x) - v_{k}(y)| )}{|x - y|^{2}} \, dy \, dx \nonumber \\[6pt] 
& +  \mathop{\int_{-n}^{n} \int_{-n}^{n}}_{|x - y|> x_{2, k} - x_{1, k}} \frac{\varphi_{\delta_k}(|v_{k}(x) - v_{k}(y)| )}{|x - y|^{2}} \, dy \, dx. 
\end{align}
By straightforward integral estimates, we have 
\begin{multline*}
 \int_{x_{1, k} - \delta_k}^{x_{1, k}} \int_{x_{1, k}}^{x_{2, k}} \frac{\varphi_{\delta_k}(|v_{k}(x) - v_{k}(y)| )}{|x - y|^{2}} \, dy \, dx \\[6pt] + \int_{x_{1, k}}^{x_{2, k}} \int_{x_{2, k}}^{x_{2, k} + \delta_k} \frac{\varphi_{\delta_k}(|v_{k}(x) - v_{k}(y)| )}{|x - y|^{2}} \, dy \, dx  
 \mathop{\le}^{\eqref{lem-C4-p1}-\eqref{lem-C4-p2}} C \delta_k, 
\end{multline*}
\begin{multline*}
 \int_{-n}^{x_{1, k} - \delta_k} \int_{x_{1, k}}^{x_{2, k}} \frac{\varphi_{\delta_k}(|v_{k}(x) - v_{k}(y)| )}{|x - y|^{2}} \, dy \, dx \\[6pt]
 + \int_{x_{1, k}}^{x_{2, k}} \int_{x_{2, k} + \delta_k}^{n} \frac{\varphi_{\delta_k}(|v_{k}(x) - v_{k}(y)| )}{|x - y|^{2}} \, dy \, dx \mathop{\le}^{\eqref{cond-varphi-1}} C \delta_k \ln n + \delta_k |\ln \delta_k|, 
\end{multline*}
$$
 \mathop{\int_{-n}^{n} \int_{-n}^{n}}_{|x - y|> x_{2, k} - x_{1, k}} \frac{\varphi_{\delta_k}(|v_{k}(x) - v_{k}(y)| )}{|x - y|^{2}} \, dy \, dx \mathop{\le}^{\eqref{cond-varphi-1}} C \delta_k \ln n, 
$$
for some positive constant $C$ independent of $k$ and $n$. It follows from \eqref{lem-C4-d1} that 
\begin{equation*}
\Lambda_{\delta_k}(v_{k}, (-n, n)) \le \Lambda_{\delta_k}(u_{k}, (0, 1) ) +  C \delta_k \ln n +  C \delta_k |\ln \delta_k|. 
\end{equation*}
This implies, for every $n$, 
\begin{equation}\label{lem-C4-e2}
\liminf_{k \to + \infty} \Lambda_{\delta_k}(u_{k}, (0, 1) ) \ge \liminf_{k \to + \infty} \Lambda_{\delta_k}(v_{k}, (-n, n)). 
\end{equation}

We have, by \eqref{lem-C4-p3},  
\begin{equation}\label{lem-C4-e3}
\|v_{k}(x) -  H_{0} (x)  \|_{L^{1}(-n, n)} \le \hat \delta_{1}^{2} n / 2   + C,  
\end{equation}
since $(\| u_k\|_{L^1(0,1)})$ is bounded. 
We now fix $n \ge 2 C / \hat \delta_1$ so that the RHS of \eqref{lem-C4-e3} is less than $n \hat \delta_{1}$ since $\hat \delta_1 < 1$.  Applying the Claim with $c=1$ and $(a_1,b_1) = (-n, n)$, we have, for large $k$, 
\begin{equation}\label{lem-C4-e4}
 \Lambda_{\delta_k}(v_{k}, (-n, n)) \ge  \gamma -\eps. 
\end{equation}
Combining  \eqref{lem-C4-e2} and \eqref{lem-C4-e4} yields 
\begin{equation*}
\liminf_{k \to + \infty} \Lambda_{\delta_k}(u_{k}, (0, 1) ) \ge   \gamma -\eps.
\end{equation*}
Since $\eps>0$ is arbitrary, \eqref{lem-C4-main} follows. 
\end{proof}

The proof of Property (G1) for $p=1$ is complete. 
 
\subsubsection{Proof of Property (G1) for $p>1$} \label{sect-G1-p}
 
 Throughout  this section, we assume that $\C > 0$ since there is nothing to prove otherwise.  The first  key ingredient of the proof is  
   
\begin{lemma}\label{lem1-G1} 
Let $p > 1$ and let $\varphi$ satisfy \eqref{cond-varphi-0}-\eqref{cond-varphi-3}.
Let $ a < b$, and $u \in L^p(a, b)$ and  let $t_1, t_2 \in (a, b)$ be two Lebesgue points of $u$. 
Then, for some positive constant $\sigma$ depending only on $\alpha$, $\beta$,  and $p$, 
\begin{equation}\label{lem1-G1-es}
 \liminf_{\delta \to 0} \Lambda_{\delta} (u_\delta, (t_1, t_2)) \ge \sigma \C (t_2 - t_1)^{1-p} |u(t_2) - u(t_1)|^p, 
\end{equation}
for any family  $(u_\delta) \subset L^p(t_1, t_2)$ such that $u_\delta \to u$ in $L^p(t_1, t_2)$, as $\delta \to 0$.
\end{lemma}

\begin{proof}  Without loss of generality, one may assume that $t_1 = 0$, $t_2 = 1$, $u(t_1) = 0$, and $u(t_2) = 1$. Let $(\delta_k)$ and $(u_k)$ be arbitrary such that $\delta_k \to 0$, $u_k \to u$ in $L^p(0, 1)$, and 
$$
 \lim_{k \to + \infty} \Lambda_{\delta_k} (u_k, (0, 1)) \mbox{ exists and is finite}. 
$$
Denote $\tau$ the limit of $\Lambda_{\delta_k} (u_k, (0, 1))$.  In order to establish \eqref{lem1-G1-es}, it suffices to prove 
\begin{equation} \label{lem1-claim}
\C \le C \tau. 
\end{equation}
Here and in what follows, $C$ denotes a positive constant  depending only on $\alpha$, $\beta$,  and $p$. 

Let $(c_k) \subset \mR_+$ be such that 
\begin{equation}\label{lem1-G1-p1}
  \lim_{k \to + \infty} c_k =0, \quad c_k \ge \delta_k^{1/2}, 
\end{equation}
\begin{equation}\label{lem1-G1-p1-1}
\Lambda_{\delta_k} (u_k, (0, 1)) \le \tau + c_k, 
\end{equation}
\begin{equation}\label{lem1-G1-p2}
\int_{0}^1 \frac{\varphi_{\delta_k}(|u_k(c_k) - u_k(y)|)}{|c_k - y|^{p+1}} \, dy  \le C c_k^{-1}(\tau + c_k), 
\end{equation}
\begin{equation}\label{lem1-G1-p3}
\int_{0}^1 \frac{\varphi_{\delta_k}(|u_k(1-c_k) - u_k(y)|)}{|1-c_k - y|^{p+1}} \, dy  \le C c_k^{-1}(\tau + c_k). 
\end{equation}
\begin{equation}\label{lem1-G1-p4}
|u_k(c_k)| + |u_k(1- c_k) - 1|  \to 0,   
\end{equation}
for large $k$.  

For simplicity of presentation, we will assume that $u_k(c_k) = 0$ and $u_k(1 - c_k) = 1$.  Define $\hat u_k: (0, 1 ) \to \mR$ as follows
\begin{equation*}
\hat u_k (x) = \left\{\begin{array}{cl}
0& \mbox{ if } 0 <   x < c_k, \\[6pt]
u_k(x) & \mbox{ if } c_k \le x \le  1 - c_k, \\[6pt] 
1 & \mbox{ if }  1 - c_k < x <  1. 
\end{array}\right.
\end{equation*}
For $n \in \mN$, set 
\begin{equation*}
f_{k, n}(x) = \hat u_k(x - [x]) + [x] \mbox{ for } x \in (0, n)
\end{equation*}
and 
\begin{equation*}
g_{k, n} (x) = \frac{1}{n} f_{k, n} (nx) \mbox{ for } x \in (0, 1).
\end{equation*}

We have, by a change of variables,  
\begin{equation*}
\Lambda_{\delta_k/n} (g_{k, n}, (0, 1)) = \frac{1}{n} \Lambda_{\delta_k} (f_{k, n}, (0, n)). 
\end{equation*}
Using \eqref{lem1-decomposition}, one can check, by straightforward integral estimates,  that 
\begin{equation*}
\Lambda_{\delta_k} (f_{k, n}, (0, n)) \le  C n (\tau + c_k) + C n \delta_k^p/ c_k^p. 
\end{equation*}
This implies, by \eqref{lem1-G1-p1} and \eqref{lem1-G1-p1-1}, 
\begin{equation}\label{lem1-p5}
 \Lambda_{\delta_k/n} (g_{k, n}, (0, 1))  \le  C \tau + C c_k.
\end{equation}

On the other hand, we have 
\begin{equation}\label{lem1-p6}
\int_{0}^1 |g_{k, n}(x) - x|^p \, dx   = \frac{1}{n^p} \int_0^n |f_{k, n} (nx) - n x|^p \, dx = 
\frac{1}{n^{p-1}} \int_0^1 |\hat u_{k}(x) - x|^p \, dx.  
\end{equation}
Taking $n = n_k =  [\ln \delta_k^{-1}]$, we derive from \eqref{lem1-p6} that 
\begin{equation*}
\lim_{k \to + \infty} \int_{0}^1 |g_{k, n_k}(x) - U|^p \, dx = 0, 
\end{equation*}
since $p>1$ and $(\|\hat u_k \|_{L^p(0, 1)})$ is bounded. By noting that $\delta_k/ n_k  \to 0$ as $k \to + \infty$, we derive from the definition of $\C$ that 
\begin{equation}\label{lem1-p7}
\C \le \liminf_{k \to + \infty} \Lambda_{\delta_k/n_k} (g_{k, n_k}, (0, 1)). 
\end{equation}
Combining \eqref{lem1-p5} and \eqref{lem1-p7} yields 
$$
\C \le C \tau, 
$$
which is \eqref{lem1-claim}. 
\end{proof}

From  Lemma~\ref{lem1-G1}, we now derive 

\begin{lemma}\label{lem1-G1-1} Let $p > 1$ and let $\varphi$ satisfy \eqref{cond-varphi-0}-\eqref{cond-varphi-3}. Let $u \in L^p(\mR)$ and assume that,  for some $(u_\delta) \subset L^p(\mR)$ converging to $u$ in $L^p(\mR)$, 
$$
\liminf_{\delta \to 0} \Lambda_{\delta}(u_\delta, \mR) < + \infty. 
$$
Then $u \in W^{1, p}(\mR)$. 
\end{lemma}

\begin{proof} As a consequence of \Cref{lem1-G1}, one has, for every $ - \infty < a< b<  + \infty$,  
\begin{equation}\label{lem1-G1-1-p1}
 \liminf_{\delta \to 0} \Lambda_{\delta} (u_\delta, (a, b)) \ge \sigma \C (b - a)^{1-p} |\mathop{\mathrm{ess \; sup} }_{x \in (a, b )}
u - \mathop{\mathrm{ess \; inf} }_{x \in (a, b)} u|^p, 
\end{equation}
for some constant $\sigma > 0$, independent of $a$ and $b$. Set, for $h \in (0, 1)$, 
$$
\tau_{h}(u)(x) =  \frac{1}{h} \Big(u(x+h) - u (x) \Big) \quad \mbox{ for } x \in \mR.
$$
For each $m \ge 2$ and $h \in (0, 1)$, fix $K > 0$ such that $Kh \ge m$. Then 
\begin{equation}\label{lem1-G1-1-p2}
\int_{-m}^m |\tau_h(u)|^p \, dx \le \sum_{k = -K}^{K}\int_{k
h}^{(k+1)h } |\tau_h(u)|^p \, dx.
\end{equation}
Since, for every $a \in \mR$, 
\begin{equation*}
\int_a^{a + h } |\tau_{h}(u)|^p \, dx \le \int_a^{a + h
}\frac{1}{h^p} |\mathop{\mathrm{ess \; sup} }_{t \in (a, a + 2h )}
u - \mathop{\mathrm{ess \; inf} }_{t \in (a, a + 2h ) } u|^p \,
dx,
\end{equation*}
it follows from \eqref{lem1-G1-1-p1} that
\begin{equation*}
\int_a^{a + h } |\tau_{h}(u)|^p \, dx \le \frac{2^{p-1}}{\sigma \C} \liminf_{\delta \to 0} \Lambda_{\delta} (u_\delta, (a, a+2h)). 
\end{equation*}
We derive from   \eqref{lem1-G1-1-p2} that 
\begin{equation*}
\int_{-m}^m |\tau_h(u)|^p \, dx \le \frac{2^{p}}{\sigma \C} \liminf_{\delta \to 0} \Lambda_{\delta} (u_\delta, \mR). 
\end{equation*}
Since $m \ge 2$ is arbitrary, we obtain, for all $h \in (0, 1)$,  
\begin{equation}\label{lem1-G1-1-p3}
\int_{\mR} |\tau_h(u)|^p \, dx \le \frac{2^{p}}{\sigma \C} \liminf_{\delta \to 0} \Lambda_{\delta} (u_\delta, \mR). 
\end{equation}
It follows that $u
\in W^{1,p}(\mR)$ (see e.g. \cite[Chapter 8]{BrAnalyse1}). \end{proof}

The second key ingredient in the proof of Property (G1) is the following useful property of functions in $W^{1, p}(\mR)$. 

\begin{lemma} \label{lem-lem} Let $p > 1$ and $u \in W^{1, p}(\mR)$ (so that $u$ admits a continuous representative still denoted by $u$). Given $\eps_1 > 0$, there exist a subset $B$ of Lebesgue points of $u'$ and  $\ell \ge 1$ such that 
\begin{equation}\label{G1-p-pro-Am}
\int_{\mR \setminus B} |u'|^p \, dx \le \eps_1 \int_{\mR} |u'|^p \, dx, 
\end{equation}
and, for every  open interval $I'$ with $|I'| \le 1/\ell$ and  $I' \cap B \neq \emptyset$,   and for every $x \in I' \cap B$, 
\begin{equation}\label{lem-lem-1}
\frac{1}{|I'|^p }\fint_{I'} |u(y) - u(x) - u'(x) (y - x)|^p \, d y \le  \eps_1  
\end{equation}
and 
\begin{equation}\label{lem-lem-2}
|u'(x)|^p \ge (1 - \eps_1) \fint_{I'} |u'(y)|^p \, dy. 
\end{equation}
\end{lemma}

\begin{proof}
We first recall the following property of $W^{1, p}(\mR)$ functions (see e.g., \cite[Theorem 3.4.2]{Ziemer}): Let $f \in W^{1,p} (\mR)$. Then,  for a.e. $x \in \mR$,
\begin{equation}\label{G1-p-measure}
\lim_{r \to 0} \frac{1}{r^p} \fint_{x-r}^{x+r} \big|f(y) - f(x) - f'(x)(y - x) \big|^p \, dy = 0.
\end{equation}

Given $n \in \mN$, define, for a.e. $x \in \mR$,
\begin{equation}\label{G1-p-def-rho}
\rho_n(x) = \sup \left\{  \frac{1}{r^{p}}
\fint_{x-r}^{x+r} \big|u(y) - u(x) - u'(x) (y-x) \big|^p \, dy ; \; r  \in (0, 1/n) \right\}
\end{equation}
and
\begin{equation}\label{G1-p-def-tau}
\tau_n(x) = \sup  \left\{ \fint_{x-r}^{x+r}| u'(y) -  u'(x)|^p \, dy;  r \in (0, 1/n)  \right\}. 
\end{equation}
Note that, by  \eqref{G1-p-measure}, $\rho_n (x) \to 0$ for a.e. $x \in \mR$ as $n \to + \infty$. We also have, $\tau_n  (x)\to 0$ for a.e. $x \in \mR$ as $n \to + \infty$ (and in fact at every Lebesgue points of $u'$). For $m \ge 1$, set 
$$
D_m = \big\{x \in (-m, m); \mbox{$x$ is a Lebesgue point of $u'$ and } |u'(x)| \ge 1/ m \big\}. 
$$
Then there exists $m \ge 1$ such that  
\begin{equation}\label{G1-p-pro-Am-1}
\int_{\mR \setminus D_m} |u'|^p \, dx \le \frac{\eps_1}{2} \int_{\mR} |u'|^p \, dx.  
\end{equation}
Fix such an $m$.  By Egorov's theorem, there exist a  subset $B$ of $ D_m$ such that  $(\rho_n)$ and $(\tau_n)$ converge to $0$ uniformly on $B$, and 
\begin{equation}\label{G1-p-pro-Am-2}
\int_{D_m \setminus B} |u'|^p \, dx \le \frac{\eps_1}{2} \int_{\mR} |u'|^p \, dx.  
\end{equation}
Combining \eqref{G1-p-pro-Am-1} and \eqref{G1-p-pro-Am-2} yields \eqref{G1-p-pro-Am}. 

We have, for every non-empty, open interval $I'$ and $x \in \mR$ (in particular for $x \in I' \cap B$),  
\begin{equation}\label{lem-lem-*}
\left(\fint_{I'} |u'(y)|^p \, dy \right)^{1/p} \le \left(\fint_{I'} |u'(y) - u'(x) |^p \, dy \right)^{1/p} + |u'(x)|, 
\end{equation}
Since $(\rho_n)$ and $(\tau_n)$ converge to $0$ uniformly on $B$ and $|u'(x)| \ge 1/ m$  for $x \in B$, it follows from \eqref{lem-lem-*} that there exists an $\ell \ge 1$ such that \eqref{lem-lem-1} and \eqref{lem-lem-2} holds. The proof is complete. 
\end{proof}


We are ready to give the 

\begin{proof}[Proof of Property (G1)]  We begin with 

{\it Claim 2:}  For $\eps > 0$, there exist two positive constants $\hat \delta_1, \hat \delta_2$ such that for every $c, d \in \mR$, for every open bounded interval  $I'$ of $\mR$, and for every
$f \in L^p(I')$ satisfying 
\begin{equation*}
\fint_{I'} |f(y) - (c y + d)|^p \, dy < \hat \delta_1 |c|^p |I'|^p,
\end{equation*}
one has
\begin{equation*}
\Lambda_{\delta} (f, I') \ge (\C - \eps)|c|^p |I'| \mbox{ for all } \delta \in (0, \hat \delta_2 |c|  |I'|). 
\end{equation*}
This claim is a consequence of the definition of $\C$ and its proof 
is omitted (it is similar to the one of Claim 1 in the proof of \Cref{lem-C4}). 

In order to  establish Property (G1), it suffices to prove that 
\begin{equation}
\liminf_{k \to + \infty} \Lambda_{\delta_k}(g_k, \mR) \ge \C \int_{\mR} |g'|^p \,dx 
\end{equation}
for every $(\delta_k) \subset \mR_+$ and $(g_k) \subset L^p(\mR)$ such that $\delta_k \to 0$ and  $g_k \to g$ in $L^p(\mR)$.

Without loss of generality, one may assume that $\liminf_{k \to + \infty} \Lambda_{\delta_k}(g_k, \mR) < + \infty$. It follows from  Lemma~\ref{lem1-G1-1} that $g \in W^{1, p}(\mR)$.  Fix $\eps > 0$ (arbitrary) and  let $\hat \delta_1$  be
the positive constant corresponding to $\eps$ in  Claim 2. Set 
$$
A_m = \{x \in \mR; \; \mbox{$x$ is a Lebesgue points of $g'$ and } |g'(x)| \le 1/m \} \mbox{ for } m \ge 1.
$$
Since 
\begin{equation*}
\lim_{m \to + \infty}  \int_{A_m} |g'|^p \, dx = 0, 
\end{equation*}
there exists $m \ge 1$ such that
\begin{equation}\label{G1-p-pro-Am-11}
\int_{A_m} |g'|^p \, dx \le \frac{\eps}{2} \int_{\mR} |g'|^p \, dx.  
\end{equation}
Fix such an $m$. 
 By Lemma~\ref{lem-lem} applied to $u = g$ and $\eps_1 = \min\{ \eps/ 2, \hat \delta_1/ (2m)^p \}$, there exist a subset $B$ of Lebesgue points of $g'$ and a positive integer $\ell$  such that 
\begin{equation}\label{estBm-11}
\int_{\mR \setminus B} | g'|^p \, dx \le \frac{\eps}{2} \int_{\mR}
|g'|^p \, dx, 
\end{equation}
and for every open interval  $I'$ with $|I'| \le 1 / \ell$ and  $I' \cap B \neq \emptyset$,  and,  for every $x \in I' \cap B$,  
\begin{equation}\label{G1-p-rho-p1}
\frac{1}{|I'|^p}
\fint_{I'} \big|g(y) - g(x) - g'(x)(y-x)  \big|^p \, dy \le  \hat \delta_1/ (2m)^p
\end{equation}
and
\begin{equation}\label{G1-p-tau-p2}
|g'(x)|^p |I'| \ge (1 - \eps) \int_{I'} |g' |^p \, dy. 
\end{equation}
Fix such an $\ell$.  Set 
$$
B_m = (B \setminus A_m)
$$
and denote 
\begin{equation*}
{\bf \Omega}_\ell = \Big\{(i/\ell, (i+1) /\ell); i \in \mZ \Big\} \quad \mbox{ and } \quad { \bf
J}_\ell = \Big\{  J \in {\bf \Omega}_\ell ; \; J
\cap B_m \neq \emptyset \Big\}.
\end{equation*}
Since $\mR \setminus (B \setminus A_m) \subset (\mR \setminus B) \cup A_m$, it follows 
from \eqref{G1-p-pro-Am-11} and \eqref{estBm-11} that 
\begin{equation}\label{estBm}
\int_{\mR \setminus B_m} | g'|^p \, dx = \int_{\mR \setminus (B \setminus A_m)} | g'|^p \, dx \le \eps \int_{\mR}
|g'|^p \, dx.  
\end{equation}

Take $J \in {\bf J}_\ell$ and $x \in J \cap B_m$. Since $g_k \to g$ in $L^p(J)$, we derive from 
\eqref{G1-p-rho-p1}  (applied with $I' = J$ which is admissible since $B_m \subset B$) that 
\begin{equation*}
\lim_{k \to + \infty} \frac{1}{|J|^p} \fint_{J} \big|g_k(y) - g(x) - g'(x)(y-x)  \big|^p \, dy \le  \hat \delta_1/ (2m)^p. 
\end{equation*}
Applying Claim 2 with $I' = J$, $f = g_k$ for large $k$, $c = g'(x)$, and $d = g(x)$, we have 
\begin{equation*}
\liminf_{k \to + \infty} \Lambda_{\delta_k}(g_k, J)  \ge
(\C- \eps) |g'(x)|^p |J|,
\end{equation*}
which implies, by \eqref{G1-p-tau-p2},
\begin{equation}\label{estI'}
\liminf_{k \to + \infty}\Lambda_{\delta_k}(g_k, J) \ge
(\C- \eps) (1- \eps)\int_{J} |g'|^p \, dy.
\end{equation}
Since
\begin{align*}
\liminf_{k \to + \infty} \Lambda_{\delta_k}(g_k, \mR) \ge  \sum_{J \in
{\bf J}_\ell }\liminf_{k \to + \infty} \Lambda_{\delta_k}(g_k, J),
\end{align*}
it follows from \eqref{estI'} that
\begin{multline*}
\liminf_{k \to + \infty} \Lambda_{\delta_k}(g_k, \mR)
\ge (\C- \eps) (1- \eps) \sum_{J \in {\bf J}_\ell} \int_{J} |g'|^p \, dx \\[6pt]
\ge (\C- \eps) (1- \eps)\int_{B_m} |g'|^p \, dx \mathop{\ge}^{\eqref{estBm}} 
 (\C- \eps) (1- \eps)^2 \int_{\mR} |g'|^p \, dx; 
\end{multline*}
here in the second inequality, we have used the fact $B_m$ is contained in $\bigcup_{J \in {\bf J}_\ell} J$ up to a null set. 
Since $\eps>0$ is arbitrary, one has
\begin{align*}
\liminf_{k \to + \infty}\Lambda_{\delta_k}(g_k, \mR) \ge
\C \int_{\mR} |g'|^p \, dx.
\end{align*}
The proof is complete. 
\end{proof}

\medskip 
\noindent {\bf Acknowledgments.} This work was completed during a visit of H.-M. Nguyen at Rutgers University. He thanks H. Brezis for the invitation and  the Department of Mathematics for its hospitality.

%
%
%
%
%

\providecommand{\bysame}{\leavevmode\hbox to3em{\hrulefill}\thinspace}
\providecommand{\MR}{\relax\ifhmode\unskip\space\fi MR }
\providecommand{\MRhref}[2]{%
  \href{http://www.ams.org/mathscinet-getitem?mr=#1}{#2}
}
\providecommand{\href}[2]{#2}

\end{document}